\documentclass{article}

\usepackage{amsmath,amssymb,amsthm}
\usepackage{bm}
\usepackage{mathrsfs}
\usepackage[all]{xy}
\usepackage{booktabs}
\usepackage{fullpage}
\usepackage{hyperref}
\usepackage{mathtools}
\usepackage{todonotes}
\hypersetup{
hidelinks,
}

\usepackage{setspace}
\usepackage[english, french]{babel}

\newcommand{\rl}{\mathbb{R}}
\newcommand{\cx}{\mathbb{C}}

\newcommand{\ai}{\sqrt{-1}}

\newcommand{\inj}{\hookrightarrow}

\newcommand{\isom}{\stackrel{\sim}{\to}}

\newcommand{\kah}{K\"ahler }

\newcommand{\ddbar}{\partial \bar{\partial}}
\newcommand{\tr}{\mathrm{tr}}

\newcommand{\prj}{\mathbb{P}}

\theoremstyle{plain}
\newtheorem{theorem}{Theorem}[section]
\newtheorem{lemma}[theorem]{Lemma}
\newtheorem{proposition}[theorem]{Proposition}

\theoremstyle{definition}
\newtheorem{definition}[theorem]{Definition}

\theoremstyle{definition}
\newtheorem{remark}[theorem]{Remark}

\begin{document}
\selectlanguage{english}

\title{Mapping properties of the Hilbert and Fubini--Study maps in K\"ahler geometry, with erratum}
\author{Yoshinori Hashimoto}

\maketitle


\selectlanguage{english}
\begin{abstract}
Suppose that we have a compact K\"ahler manifold $X$ with a very ample line bundle $\mathcal{L}$. We prove that any positive definite hermitian form on the space $H^0 (X,\mathcal{L})$ of holomorphic sections can be written as an $L^2$-inner product with respect to an appropriate hermitian metric on $\mathcal{L}$. We apply this result to show that the Fubini--Study map, which associates a hermitian metric on $\mathcal{L}$ to a hermitian form on $H^0 (X,\mathcal{L})$, is injective.
\end{abstract}

\section*{Erratum}

One of the main results of this paper, the surjectivity of the Hilbert map (in Theorem 1.1), is wrong. A counterexample was provided by Jingzhou Sun \cite[\S 3, Example]{Sun22}. Lemma 3.1, whose proof rests on the surjectivity of the Hilbert map, is also wrong. A counterexample was provided by L{\'a}szl{\'o} Lempert in an email correspondence with the author (see \cite[Proposition 4.16]{Finski22}). A precise statement on the closure of the image of the Hilbert map was given by Sun \cite[Theorems 1.2 and 1.3]{Sun22}.

The other result in Theorem 1.1, the injectivity of the Fubini--Study map, does hold, although the proof in this paper does not hold since Lemma 3.1 is wrong as stated above. An independent proof of this theorem, together with a more precise statement for the image of the Fubini--Study map, was given by Lempert \cite[Theorem 1.1]{Lem21}.

The author is so far unable to exactly point out which part of the proof in this paper breaks down, but both counterexamples mentioned above involve basis elements that are redundant in terms of the global generation of the line bundle (see the definition of a rather ample subspace in \cite[\S 6]{Lem21}).

A corrected, weaker version of Lemma 3.1 is available in an arXiv preprint \cite{yhqifs} by the author, but he believes that it is better to separate it as a new work because of its length, and because its statement is different and the proof involves ideas different from the ones used in this paper. Thus its details are not discussed in this erratum. The author sincerely apologises for the time it took to prepare this erratum, but it was necessary for him to find a corrected version of Lemma 3.1 and prove it.

The author thanks Siarhei Finski, L{\'a}szl{\'o} Lempert, and Jingzhou Sun for very helpful discussions and comments. This work is partially supported by JSPS KAKENHI (Grant-in-Aid for Scientific Research (C)), Grant Number JP23K03120, and JSPS KAKENHI (Grant-in-Aid for Scientific Research (B)) Grant Number JP24K00524.

\begin{flushleft}
{\footnotesize
Department of Mathematics, Osaka Metropolitan University, \\
3-3-138, Sugimoto, Sumiyoshi-ku, Osaka, 558-8585, Japan.\\
Email: \texttt{yhashimoto@omu.ac.jp}}
\end{flushleft}

\section{Introduction and statement of the result}

Let $(X,\mathcal{L})$ be a polarised \kah manifold of complex dimension $n$. We have
\begin{itemize}
	\item[$\bullet$] $\mathcal{H} (X,\mathcal{L}) :=$ the set of all positively curved hermitian metrics on $\mathcal{L}$,
	\item[$\bullet$]  $\mathcal{B} :=$ the set of all positive definite hermitian forms on $H^0 (X,\mathcal{L})$,
\end{itemize}
where $\mathcal{H} (X,\mathcal{L})$ is infinite dimensional and $\mathcal{B}$ is finite dimensional.

We can define the following two maps;
\begin{itemize}
	\item[$\bullet$] the \textbf{Hilbert map} $Hilb : \mathcal{H} (X,\mathcal{L}) \to \mathcal{B}$ defined by the $L^2$-inner product of $h \in \mathcal{H} (X,\mathcal{L})$,
	\item[$\bullet$] the \textbf{Fubini--Study map} $FS :  \mathcal{B} \to \mathcal{H} (X,\mathcal{L}) $ defined as the pullback of the Fubini--Study metric on the projective space $\prj (H^0 (X,\mathcal{L})^*)$.
\end{itemize}

The result that we prove in this paper is the following.
\begin{theorem}
	\label{hilbsjfsinj}
	Suppose that $\mathcal{L}$ is very ample. Then $Hilb$ is surjective and $FS$ is injective.
\end{theorem}

Since $Hilb$ is a map from an infinite dimensional manifold to a finite dimensional manifold, it seems natural to speculate that it is surjective. Similarly, it also seems natural to speculate that $FS$ is injective. Indeed, these statements seem to be widely believed among the experts in the field. However, the proof of these facts do not seem to be explicitly written in the literature previously, to the best of the author's knowledge; in fact, as we shall see, the proof that we give involves application of the degree of continuous maps and has connection to the Aubin--Yau theorem \cite{aubin,yaucy}. We shall provide the proof of these ``folklore'' statements in this paper.

For $FS$, we shall in fact prove a stronger quantitative result for the injectivity (cf.~Lemma \ref{lemfsinj}), which was applied in \cite{yhpreprint} to find a point in $\mathcal{B}$ that is close to the minimum of the modified balancing energy. Generalisations to several variants of the $Hilb$ map (cf.~Proposition \ref{lemsurjhilbnu}) will also be discussed in \S \ref{secvarhilb}.

\section*{Acknowledgements}
\begin{small}
Part of this work was carried out in the framework of the Labex Archim\`ede (ANR-11-LABX-0033) and of the A*MIDEX project (ANR-11-IDEX-0001-02), funded by the ``Investissements d'Avenir" French Government programme managed by the French National Research Agency (ANR). Much of this work was carried out when the author was a PhD student at the Department of Mathematics of the University College London, which he thanks for the financial support. Theorem \ref{hilbsjfsinj} and its proof form part of the author's PhD thesis submitted to the University College London.

The author thanks S\'ebastien Boucksom, Julien Keller and Jason Lotay for helpful discussions and comments. He is grateful to the anonymous referee for helpful suggestions that significantly improved the exposition of this paper, and for pointing out Lemma \ref{remphzrid}.
\end{small}

\section{Surjectivity of $Hilb$} \label{pfofsjtvthilb}

\subsection{Preliminary definitions} \label{secprelimdef}
Suppose that $\mathcal{L}$ is very ample. We shall first define the maps $Hilb$ and $FS$ more precisely as follows (cf.~\cite{donproj2}).

\begin{definition} \label{defhilbfs}
	The \textbf{Hilbert map} $Hilb : \mathcal{H} (X,\mathcal{L}) \to \mathcal{B}$ is defined by
\begin{equation*}
Hilb (h) : = \int_X h ( , ) \omega_h^n .
\end{equation*}

The \textbf{Fubini--Study map} $FS :  \mathcal{B} \to \mathcal{H} (X,\mathcal{L}) $ is defined by the equation
\begin{equation}
\sum_{i=1}^N |s_i|^2_{FS(H)} =1 \label{defoffseq}
\end{equation}
where $\{ s_i \}_i$ is an $H$-orthonormal basis for $H^0 (X , \mathcal{L})$ and $N:= \dim_{\cx} H^0(X , \mathcal{L})$.
\end{definition}

\begin{remark} \label{remscalhbfs}
	The above definition is slightly different from the ones that are often used in the literature by a factor of scaling, in the following sense. Often we consider the tensor power $\mathcal{L}^{\otimes k}$, and the usual definition of $Hilb$ concerns a hermitian metric $h$ on $\mathcal{L}$ (rather than $\mathcal{L}^{\otimes k}$), with $Hilb (h) : = \frac{N_k}{V} \int_X h^{\otimes k} ( , ) \frac{\omega_h^n}{n!}$ defining a hermitian form on the vector space $H^0 (X, \mathcal{L}^{\otimes k})$ of dimension $N_k$, with $V:= \int_X c_1(L)^n / n!$. For a hermitian form $H$ on $H^0 (X, \mathcal{L}^{\otimes k})$, $FS(H)$ is a hermitian metric on $\mathcal{L}$ satisfying $\sum_i |s_i|^2_{FS(H)^{\otimes k}} =1$.
	
	As far as Theorem \ref{hilbsjfsinj} is concerned, these differing conventions will not cause any problem, and we shall use the one given above in this paper not to be concerned with extra exponent of $k$ or the constant $N_k$. For the quantitative result Lemma \ref{lemfsinj}, however, this does require some straightforward modification which is mentioned in Remark \ref{remscalfsquan}. We shall consider $\mathcal{L}^{\otimes k}$ also in \S \ref{secvarhilb}.
\end{remark}
 
We also recall (and slightly modify) some definitions made by Bourguignon--Li--Yau \cite[\S 2]{bly} that we shall use in the proof of surjectivity of $Hilb$ in \S \ref{secpfsurjhilb}. Since $\mathcal{L}$ is very ample we have the Kodaira embedding $\iota : X \inj \prj(H^0(X,\mathcal{L})^*) \isom \prj^{N -1}$. First of all pick homogeneous coordinates $\{ Z_i \}_i$ on $\prj^{N-1}$; all matrices appearing in what follows will be with respect to this basis $\{ Z_i \}_i$, unless otherwise specified. This then defines a hermitian metric $\tilde{h} := \tilde{h}_{FS( \mathrm{Id} )}$ on $\mathcal{O}_{\prj^{N-1}} (1)$ and the Fubini--Study metric $\omega_{\widetilde{FS} ( \mathrm{Id} )}$ on $\prj^{N-1}$, where $\mathrm{Id}$ stands for the identity matrix. We write $d {\mu}_Z$ for the volume form on $\prj^{N-1}$ defined by $\omega_{\widetilde{FS} ( \mathrm{Id} )}$.

Suppose that we pick $B \in GL(N , \cx)$ and change the basis from $\{ Z_i \}_i$ to  $BZ := \{ Z'_i \}_i$ (where $Z'_i := \sum_j B_{ij} Z_j$). This defines a new Fubini--study metric $\omega_{\widetilde{FS}(H)}$ associated to the hermitian metric $\tilde{h}_{FS( H )}$, where $H := \overline{(B^{-1})^t} B^{-1}$; note that $H$ has $BZ $ as its orthonormal basis. We also write $d \mu_{BZ}$ for the volume form corresponding to $\omega_{\widetilde{FS}(H)}$.

\begin{definition}
	We define two spaces of hermitian matrices as follows:
	\begin{itemize}
		\item[$\bullet$] $\mathcal{J}^{\circ} : = \{ N \times N \text{ positive definite hermitian matrices with operator norm } 1 \} ,$
		\item[$\bullet$] $\mathcal{H}^{\circ} := \{ N \times N \text{ positive definite hermitian matrices with trace } 1 \}.$
	\end{itemize}
	
\end{definition}

Since $B \in \mathcal{J}^{\circ}$ is positive definite hermitian, the operator norm $|| B ||_{op}$ of $B$ being 1 is equivalent to the largest eigenvalue of $B$ being 1 (all eigenvalues of $B$ are real and positive). Observe also that any $N \times N$ positive definite hermitian matrix $B'$ can be divided by a positive real constant $\alpha_1 := ||B'||_{op}$ (resp.~$\alpha_2 := \tr(B')$) to be an element of $\mathcal{J}^{\circ}$ (resp.~$\mathcal{H}^{\circ}$).

In other words, $\mathcal{J}^{\circ}$ and $\mathcal{H}^{\circ}$ just denote two different scalings of the set of $N \times N$ positive definite hermitian matrices, and hence they are diffeomorphic under the map $\mathcal{J}^{\circ} \ni B \mapsto B / \tr(B) \in \mathcal{H}^{\circ}$. Artificial as it may seem, it will be useful in \S \ref{secpfsurjhilb} to fix the scalings as above.

Note $\dim_{\rl} \mathcal{J}^{\circ} = \dim_{\rl} \mathcal{H}^{\circ} =  N^2 -1$ and that $\mathcal{J}^{\circ}$ and $\mathcal{H}^{\circ}$ can be identified with connected bounded open subsets in $\rl^{N^2 -1}$. We endow them with the Euclidean topology inherited from $\rl^{N^2 -1}$.

With respect to this topology, we can compactify $\mathcal{J}^{\circ}$ (resp.~$\mathcal{H}^{\circ}$) to $\mathcal{J}$ (resp.~$\mathcal{H}$) by adding a topological boundary $\partial \mathcal{J}$ (resp.~$\partial \mathcal{H}$) defined as
\begin{itemize}
	\item[$\bullet$] $\partial \mathcal{J} : =  \{ N \times N \text{ positive semi-definite hermitian matrices with operator norm } 1 \text{ and rank} \le N-1 \}$,
	\item[$\bullet$] $\partial \mathcal{H} := \{ N \times N \text{ positive semi-definite hermitian matrices with trace } 1 \text{ and rank} \le N-1 \}$.
\end{itemize}

\subsection{Degree of continuous maps}

We summarise some well-known results on the degree of continuous maps between domains in Euclidean spaces. The reader is referred to \cite[Chapter 12]{amd} for more details. Let $\Omega$ be a bounded open subset in $\rl^m$ and $\bar{\Omega}$ be its closure in $\rl^m$ with respect to the Euclidean topology. We also write $\partial \Omega : = \bar{\Omega} \setminus \Omega$ for the boundary.

We first consider a $C^1$-map $\phi \in C^1 (\bar{\Omega} , \rl^m )$. Writing componentwise in terms of $\rl^m$, we can write $\phi = (\phi_1 , \dots , \phi_m)$. The \textbf{Jacobian} of $\phi$ at $x$ is defined as $J_{\phi}(x) := \det \left( \frac{\partial \phi_i}{\partial x_j} (x) \right)$. Recall that $x \in \bar{\Omega}$ is called a \textbf{critical point} of $\phi \in C^1 (\bar{\Omega} , \rl^m)$ if $J_{\phi} (x) = 0$. If $x$ is a critical point of $\phi$, $\phi (x)$ is called a \textbf{critical value}. Suppose that $\phi \in C^1 (\bar{\Omega} , \rl^m )$, $p \notin \phi (\partial \Omega )$, and $p$ is not a critical value of $\phi$. The \textbf{degree} of $\phi$ at $p$ relative to $\Omega$ is defined by
	\begin{equation} \label{defdegcone}
		\mathrm{deg} (\phi , \Omega , p) := \sum_{x \in \phi^{-1} (p)} \mathrm{sign} J_{\phi} (x).
	\end{equation}
Note $\mathrm{deg} (\phi , \Omega , p) =0$ if $\phi^{-1} (p) = \emptyset$.

If $p$ is a critical value of $\phi$, we define $\mathrm{deg} (\phi , \Omega , p)$ by slightly perturbing $p$ (cf.~\cite[Definition 12.4]{amd}). Moreover, for a continuous map $Q \in C^0 (\bar{\Omega} , \rl^m )$, we take a ``perturbation'' $\phi_{Q,p} \in C^1 (\bar{\Omega} , \rl^m )$ of $Q$ to define $\mathrm{deg} (Q , \Omega , p)$ as $\mathrm{deg} (\phi_{Q,p} , \Omega , p)$ (cf.~\cite[Definition 12.5]{amd}). These are well-defined, since they do not depend on the perturbations chosen \cite[Theorems 12.3-12.9]{amd}.

An important result is the homotopy invariance of the degree.

\begin{theorem} \emph{(cf.~\cite[Theorem 12.11]{amd})} \label{htpyinvdeg}
	Let $Q_0 , Q_1 \in C^0 (\bar{\Omega} , \rl^m)$ and $Q_t$ be a homotopy in $C^0 (\bar{\Omega} , \rl^m)$ between $Q_0$ and $Q_1$. Let $p \in \rl^m$ such that $p \notin Q_t (\partial \Omega )$ for all $t \in [0,1]$. Then $\mathrm{deg} (Q_0 , \Omega , p) = \mathrm{deg} (Q_1 , \Omega , p)$.
\end{theorem}

Another important result is that nontriviality of degree ``detects'' surjectivity.

\begin{theorem} \emph{(cf.~\cite[Theorem 12.10]{amd})} \label{nontrvdegsurj}
	Let $Q \in C^0 (\bar{\Omega} , \rl^m)$ and $p \notin Q (\partial \Omega )$. If $\mathrm{deg} (Q , \Omega , p) \neq 0$ then there exists $x \in \Omega$ such that $Q(x) = p$.
\end{theorem}

In what follows, we shall take $\Omega = \mathcal{J}^{\circ}$ and $\bar{\Omega} = \mathcal{J}$, and define a map $Q_X \in C^0(\bar{\Omega} , \rl^{N^2 -1})$ with $Q_X (\bar{\Omega}) \subset \mathcal{H}$ and $Q_X (\partial \Omega) \subset \partial \mathcal{H}$ (note that this implies $Q_X : \mathcal{J}^{\circ} \to \mathcal{H}^{\circ}$ is proper). We shall apply the above theorems to claim surjectivity of $Q_X$, which in turn proves the surjectivity of $Hilb$.

\subsection{Proof of surjectivity of $Hilb$} \label{secpfsurjhilb}

We now prove the first part of Theorem \ref{hilbsjfsinj}, i.e.~surjectivity of $Hilb$. After several technical definitions and lemmas, the main surjectivity result will be proved in Proposition \ref{lemsurjhilb}. The main line of the argument presented below is similar to Bourguignon--Li--Yau \cite[\S 2]{bly}.

\begin{definition} \label{defqprj}
	We define a continuous map $Q_{\prj} : \mathcal{J}^{\circ} \to \mathcal{H}^{\circ}$ as
\begin{equation*}
Q_{\prj} (B)_{ij} :=  \left( \int_{\prj^{N-1}} \frac{\sum_l |Z_l|_{\tilde{h}}^2}{\sum_l \left| \sum_m B_{lm} Z_m \right|_{\tilde{h}}^2} d \mu_{BZ} \right)^{-1} \int_{ \prj^{N-1}} \frac{\tilde{h} (Z_i ,  {Z}_j) }{\sum_l \left| \sum_m B_{lm} Z_m \right|^2_{\tilde{h}}} d \mu_{BZ},
\end{equation*}
where $Q_{\prj} (B)_{ij} $ stands for the $(i,j)$-th entry of the matrix $Q_{\prj} (B) \in \mathcal{H}^{\circ}$.
\end{definition}

Writing $\xi_B: \prj^{N-1} \isom \prj^{N-1}$ for the biholomorphic map induced from the linear action of $B \in \mathcal{J}^{\circ}$ on $\prj^{N-1}$, we note
\begin{align*}
Q_{\prj} ( \mathrm{Id} )_{ij} &= \left( \int_{\prj^{N-1} }  d \mu _Z \right)^{-1} \int_{ \prj^{N-1}} \frac{\tilde{h} (Z_i ,  {Z}_j) }{\sum_l \left|   Z_l \right|^2_{\tilde{h}}} d \mu_{Z} \\
&= \left( \int_{\prj^{N-1} }  (\xi_B^* d \mu_{Z}) \right)^{-1} \int_{ \prj^{N-1}} \xi_B^* \left( \frac{\tilde{h} (Z_i ,  {Z}_j) }{\sum_l \left|   Z_l \right|^2_{\tilde{h}}}  \right) (\xi_B^* d \mu_{Z}) \\
&= \left( \int_{\prj^{N-1} }  d \mu _{BZ} \right)^{-1}  \int_{ \prj^{N-1}} \frac{\sum_{l,m} \tilde{h} (B_{il} Z_l ,  B_{jm} Z_m) }{\sum_l \left| \sum_m B_{lm} Z_m \right|^2_{\tilde{h}}} d \mu_{BZ}
\end{align*}
and hence, recalling $\tr (Q_{\prj} (B)) =1$ and writing $B^t$ for the transpose of $B$, we get
\begin{equation} \label{defofpsizrb}
Q_{\prj} (B) = \frac{(B^t)^{-1} Q_{\prj} ( \mathrm{Id} ) (B^t)^{-1}}{ \tr ((B^t)^{-1} Q_{\prj} ( \mathrm{Id} ) (B^t)^{-1})},
\end{equation}
by recalling $B = B^*$ which implies $B^t = \bar{B}$.

We have the following lemma.

\begin{lemma}\label{remphzrid}
	$Q_{\prj} ( \mathrm{Id} )$ is a constant multiple of the identity matrix.
\end{lemma}

\begin{proof}
	We define a local coordinate system $(z_1 , \dots , z_{N-1})$ on a Zariski open subset $\mathcal{U}$ of $\prj^{N-1}$ defined by $Z_N \neq 0$, so that $z_i = Z_i / Z_N$. We write $z_i$ in polar coordinates as $z_i = r_i e^{\ai \theta_i}$.
	
	Assuming $i, j \neq N$, we get
	\begin{equation*}
		\frac{\tilde{h} (Z_i ,  {Z}_j) }{\sum_{l=1}^N \left|   Z_l \right|^2_{\tilde{h}}} = \frac{\bar{z}_i z_j}{1+ \sum_{l=1}^{N-1} |z_l|^2} = \frac{r_i r_j e^{\ai (\theta_j - \theta_i)}}{1+\sum_{l=1}^{N-1} r_l^2}.
	\end{equation*}
	Recall also that in this coordinate system, the Fubini--Study volume form $d \mu_Z$ can be written as
	\begin{equation*}
		d \mu_Z = \frac{1}{\left( 1+ \sum_{l=1}^{N-1} r_l^2 \right)^N} \prod_{i=1}^{N-1} \left( \frac{r_i}{ \pi} d r_i \wedge d \theta_i \right).
	\end{equation*}
	Thus, for $i, j \neq N$, we get the integral (since $\mathcal{U}$ is Zariski open in $\prj^{N-1}$) as
	\begin{equation*}
		Q_{\prj} (\mathrm{Id})_{ij} = \int_{\mathcal{U}} \frac{\tilde{h} (Z_i ,  {Z}_j) }{\sum_l \left|   Z_l \right|^2_{\tilde{h}}} d \mu_Z = \int_0^{2 \pi}  \int_0^{2 \pi}  e^{\ai (\theta_j - \theta_i )} d \theta_i d \theta_j \left( \int_0^{\infty} \dots \int_0^{\infty} \frac{(2 \pi)^{N-3} r_i r_j}{(1+\sum_{l=1}^{N-1} r_l^2)^{N+1}} \prod_{m=1}^{N-1} \frac{r_m dr_m}{\pi} \right),
	\end{equation*}
	which is non-zero if and only if $i=j$ because of the integral in $\theta_i$ and $\theta_j$.
	
	Exactly the same reasoning shows that $Q_{\prj}(\mathrm{Id})_{iN} = Q_{\prj}(\mathrm{Id})_{Nj}=0$ for $i, j \neq N$; we only need to observe that, in this case, we have
	\begin{equation*}
		\frac{\tilde{h} (Z_i ,  {Z}_N) }{\sum_{l=1}^N \left|   Z_l \right|^2_{\tilde{h}}} = \frac{\bar{z}_i}{1+ \sum_{l=1}^{N-1} |z_l|^2} = \frac{r_i e^{- \ai \theta_i }}{1+\sum_{l=1}^{N-1} r_l^2},
	\end{equation*}
	which becomes zero under the integration in $\theta_i$, as we did above.
	
	Thus the only nonzero entries of $Q_{\prj}$ are the diagonal ones. On the other hand, we observe that the symmetry in exchanging the variables $Z_1 , \dots , Z_N$ imply that the diagonal entries of $Q_{\prj}$ must be all equal. Thus $Q_{\prj}$ must be a constant multiple of the identity.
\end{proof}

\begin{lemma} \label{lempsizediffoss}
	$Q_{\prj}$ defines a diffeomorphism between $\mathcal{J}^{\circ}$ and $\mathcal{H}^{\circ}$.
\end{lemma}

\begin{proof}
	By Lemma \ref{remphzrid} we have $Q_{\prj} (B) = (B^t)^{-2} / \tr ((B^t)^{-2})$. By recalling the general fact that a positive definite hermitian matrix $A$ has a unique positive square root $A^{1/2}$, it easily follows that $Q_{\prj}$ is bijective.
	
	Since $Q_{\prj}$ is smooth and the map $A \mapsto A^{1/2}$ is smooth as long as $A$ is positive definite hermitian, we see that $Q_{\prj} : \mathcal{J}^{\circ} \to \mathcal{H}^{\circ}$ is a diffeomorphism.
\end{proof}

\begin{lemma} \label{lembdrpsizr}
$Q_{\prj} : \mathcal{J}^{\circ} \to \mathcal{H}^{\circ}$ extends continuously to the boundary of $\mathcal{J}$, and sends elements of $\partial \mathcal{J}$ into the ones in $\partial \mathcal{H}$, i.e.~$Q_{\prj} \in C^0(\mathcal{J} , \mathcal{H} )$ with $Q_{\prj} (\partial \mathcal{J} ) \subset \partial \mathcal{H}$.
\end{lemma}

\begin{proof}
Suppose that we are given a sequence $\{ B_{\nu} \}_{\nu} \subset \mathcal{J}^{\circ}$ converging to a point $B_{\infty}$ in $\partial \mathcal{J}$ (in the Euclidean topology, as remarked at the end of \S \ref{secprelimdef}) as $\nu \to \infty$. Let $\lambda_{\nu} >0$ be the smallest eigenvalue of $B_{\nu}$. Since $B_{\nu} \to \partial \mathcal{J}$ as $\nu \to \infty$ and the operator norm of $B_{\nu}$ is fixed to be 1, we have $\lambda_{\nu} \to 0$.

Observe, by homogeneity, that we have
\begin{align*}
	Q_{\prj} (B_{\nu}) &= \frac{(B_{\nu}^t)^{-1} Q_{\prj} ( \mathrm{Id} ) (B_{\nu}^t)^{-1}}{ \tr ((B_{\nu}^t)^{-1} Q_{\prj} ( \mathrm{Id} ) (B_{\nu}^t)^{-1})} \\
	&= \frac{\lambda_{\nu} ( B_{\nu}^t)^{-1} Q_{\prj} ( \mathrm{Id} ) \lambda_{\nu} ( B_{\nu}^t)^{-1}}{ \tr ( \lambda_{\nu} (B_{\nu}^t)^{-1} Q_{\prj} ( \mathrm{Id} ) \lambda_{\nu} ( B_{\nu}^t)^{-1})} .
\end{align*}

We shall prove in Lemma \ref{lemconvbeta} that $(\lambda_{\nu}^{-1} B_{\nu})^{-1} = \lambda_{\nu} B_{\nu}^{-1}$ converges to a well-defined positive semidefinite hermitian matrix $\beta_{\infty}$ that is strictly semidefinite (i.e.~with rank $1 \le \mathrm{rk} (\beta_{\infty}) \le N-1$). Assuming this result, the above equation means that $(\lambda_{\nu}^{-1} B^t_{\nu})^{-1} Q_{\prj} ( \mathrm{Id} ) (\lambda_{\nu}^{-1} B^t_{\nu})^{-1}$ (the numerator of $Q_{\prj} (B_{\nu})$) converges to a well-defined positive semidefinite hermitian matrix as $\nu \to \infty$, which is strictly semidefinite since $\displaystyle{\beta_{\infty} = \lim_{\nu \to \infty} (\lambda_{\nu}^{-1} B_{\nu})^{-1}}$ is.

Thus we see that $Q_{\prj} (B_{\nu})$ converges to a well-defined positive semidefinite hermitian matrix as $\nu \to \infty$. Hence we get
\begin{equation*}
	Q_{\prj} \left( \lim_{\nu \to \infty} B_{\nu} \right) =  \frac{ \beta_{\infty} Q_{\prj} ( \mathrm{Id} ) \beta_{\infty}}{\tr( \beta_{\infty} Q_{\prj} ( \mathrm{Id} )\beta_{\infty} ) } = \lim_{\nu \to \infty} Q_{\prj} (B_{\nu}).
\end{equation*}
Moreover, since $\beta_{\infty}$ is strictly semidefinite, we see that the limit is an element in $\partial \mathcal {H}$, i.e.~$Q_{\prj} (\partial \mathcal{J}) \subset \partial \mathcal{H}$.

We finally note that any element in $\partial \mathcal{J}$ can be realised as the limit of a sequence $\{ B_{\nu} \}_{\nu} \subset \mathcal{J}^{\circ}$, which establishes the continuous extension $Q_{\prj} \in C^0(\mathcal{J} , \mathcal{H} )$ that we claimed, granted Lemma \ref{lemconvbeta} that we prove below.
\end{proof}

\begin{lemma} \label{lemconvbeta}
	Suppose that we are given a sequence $\{ B_{\nu} \}_{\nu} \subset \mathcal{J}^{\circ}$ converging to a point $B_{\infty}$ in $\partial \mathcal{J}$. Let $\lambda_{\nu} >0$ be the smallest eigenvalue of $B_{\nu}$. Then $(\lambda_{\nu}^{-1} B_{\nu})^{-1} = \lambda_{\nu} B_{\nu}^{-1}$ converges to a well-defined positive semidefinite hermitian matrix $\beta_{\infty}$ which is strictly semidefinite, i.e.~its rank satisfies $1 \le \mathrm{rk} (\beta_{\infty}) \le N-1$.
\end{lemma}

\begin{proof}
Since $B_{\nu} \to B_{\infty}$ as $\nu \to \infty$, we see that each eigenvalue (ordered, counted with multiplicities) of $\lambda_{\nu} B_{\nu}^{-1}$ converges. Now, since the unitary group $U(N)$ is compact, we may pick a subsequence $\{ \nu (i) \}_i \subset \{ \nu \}_{\nu}$ such that unitarily diagonalising bases for $B_{\nu (i)}$'s converge; more precisely, this means the existence of a convergent sequence of bases $\{ Z_{1,i} \}_i$ (identified with a convergent sequence in $U(N)$) such that $B_{\nu (i)}$ is diagonal with respect to $Z_{1,i}$, whose entries are the eigenvalues of $B_{\nu (i)}$. For this subsequence $\{ \nu (i) \}_i$ we have a well-defined limit $\displaystyle{\beta_{1,\infty}: = \lim_{i \to \infty} \lambda_{\nu (i)} B_{\nu (i)}^{-1}}$, since both the eigenvalues and the diagonalising bases for $\lambda_{\nu (i)} B_{\nu (i)}^{-1}$ converge. Suppose that we take another such subsequence $\{ \nu (j) \}_{j}$ with the limit $\beta_{2,\infty}$, where each $B_{\nu (j)}$ is diagonal with respect to a basis $Z_{2,j}$, whose entries are the eigenvalues of $B_{\nu (j)}$.

We shall assume without loss of generality that eigenvalues are always ordered and counted with multiplicities, when a matrix is (unitarily) diagonalised.

Let $u_{i,j} \in U(N)$ be a change of basis matrix from $Z_{1,i}$ to $Z_{2,j}$ with the hermitian conjugate $u_{i,j}^*$, which implies that $u_{i,j}^* \lambda_{\nu (j)} B^{-1}_{\nu (j)} u_{i,j}$ is diagonal with respect to the basis $Z_{1,i}$. A crucially important point is that $\{ u_{i,j} \}_{i,j}$ converges as $i,j \to \infty$, since $\{ Z_{1,i} \}_i$ and $\{ Z_{2,j} \}_j$ converge. Thus, by taking the limit $i \to \infty$ and $j \to \infty$, we see that $\beta_{1,\infty} = \hat{u}^* \beta_{2,\infty} \hat{u}$ for a well-defined unitary matrix $\hat{u} := u_{\infty , \infty} \in U(N)$.

Our aim is to prove $\beta_{1,\infty } =  \beta_{2,\infty }$. We now take the basis $Z_{1,i}$ as above, which unitarily diagonalises $B_{\nu (i)}$. Recall also that $B_{\nu (j)}$ is unitarily diagonalised by another basis $Z_{2,j}$, where two bases $Z_{1,i}$ and $Z_{2,j}$ are related by $u_{i,j} \in U(N)$ as $Z_{2,j} = u_{i,j} \cdot Z_{1,i}$. Thus, $B_{\nu (i)}$ and $u_{i,j}^* B_{\nu (j)} u_{i,j}$ are both diagonal with respect to $Z_{1,i}$, whose entries are their eigenvalues (ordered and counted with multiplicities). Since the eigenvalues of $\{ B_{\nu} \}_{\nu}$ converge, for any $\epsilon >0$ there exists $\nu' \in \mathbb{N}$ such that
\begin{equation*}
	 || B_{\nu (i)} - u_{i,j}^* B_{\nu (j)} u_{i,j} ||_{op} < \epsilon
\end{equation*}
holds for all $\nu (i), \nu(j) > \nu'$. Note on the other hand
\begin{equation*}
	B_{\nu (i)} - u_{i,j}^* B_{\nu (j)} u_{i,j} = B_{\nu (i)} - u_{i,j}^* B_{\nu (i)} u_{i,j} +  u_{i,j}^*(B_{\nu (i)} -B_{\nu (j)}) u_{i,j}.
\end{equation*}
Since $\{ B_{\nu} \}_{\nu}$ itself converges, for any $\epsilon >0$ there exists $\nu' \in \mathbb{N}$ such that
\begin{equation*}
	|| B_{\nu (i)} - u_{i,j}^* B_{\nu (i)} u_{i,j} ||_{op} \le || B_{\nu (i)} - u_{i,j}^* B_{\nu (j)} u_{i,j}||_{op} +  ||u_{i,j}^*(B_{\nu (i)} -B_{\nu (j)}) u_{i,j} ||_{op} < 2 \epsilon
\end{equation*}
holds for all $\nu (i), \nu(j) > \nu'$, by also noting that unitary matrices preserve the operator norm. Since $u_{i,j}$ converges to $\hat{u} \in U(N)$ as $i,j \to \infty$, we take this limit in the above to conclude $B_{1,\infty} = \hat{u}^* B_{1,\infty} \hat{u}$ with respect to the basis $Z_{1,\infty}$, where $\displaystyle{B_{1,\infty} := \lim_{i \to \infty} B_{\nu (i)}}$. Recall that $B_{\nu(i)}$ is diagonal with respect to $Z_{1,i}$ and $B_{1,\infty}$ is diagonal with respect to $Z_{1, \infty}$, by definition. Hence $\hat{u}$ can be written as a block diagonal matrix with respect to $Z_{1, \infty}$ as $\hat{u}= \mathrm{diag}(\hat{u}_1 , \dots , \hat{u}_m)$, where each $\hat{u}_l$ corresponds to the unitary matrix acting on the $l$-th eigenspace (including $\ker B_{1,\infty}$) of $B_{1,\infty}$. Recall that $\displaystyle{\beta_{1,\infty}: = \lim_{i \to \infty} \lambda_{\nu (i)} B_{\nu (i)}^{-1}}$ is diagonal with respect to the basis $Z_{1 , \infty}$ and has the same eigenspace decomposition as $B_{1,\infty}$. In particular, $\hat{u}  \beta_{1,\infty } \hat{u}^* =\beta_{1,\infty}$. Thus, combined with our hypothesis $\beta_{1,\infty} = \hat{u}^* \beta_{2,\infty }\hat{u}$, we get $\beta_{1,\infty} =  \beta_{2,\infty}$ as claimed.

We thus conclude that the limit $\beta_{\infty}:= \displaystyle{\lim_{\nu \to \infty} \lambda_{\nu} B_{\nu}^{-1}}$ is well-defined, independently of the subsequence chosen. This is positive semidefinite hermitian since it is a limit of positive definite hermitian matrices, and has rank $1 \le \mathrm{rk} (\beta_{\infty}) \le N-1$ since $\lambda_{\nu}$ is the minimum eigenvalue of $B_{\nu}$ that tends to zero as $\nu \to \infty$.
\end{proof}

Recall the Kodaira embedding $\iota : X \inj \prj(H^0(X,\mathcal{L})^*) \isom \prj^{N -1}$ and the volume form $d \mu_{BZ}$ on $\prj^{N-1}$, defined with respect to $B \in GL(N , \cx)$ and homogeneous coordinates $\{ Z_i \}_i$ as given in \S \ref{secprelimdef}.

Now suppose that we write $ \iota^*_X (d \mu_{BZ})$ for the volume form on $\iota(X) \subset \prj^{N-1}$, defined by the Fubini--Study metric associated to the hermitian matrix $H$ that has $BZ$ as its orthonormal basis (in the notation of \S \ref{secprelimdef}). Analogously to Definition \ref{defqprj}, we define the following.

\begin{definition} \label{defqx}
	We define a continuous map $Q_X : \mathcal{J}^{\circ} \to \mathcal{H}^{\circ}$ as
\begin{equation*}
Q_X (B)_{ij} :=  \left( \int_{\prj^{N-1}} \frac{\sum_l |Z_l|_{\tilde{h}}^2}{\sum_l \left| \sum_m B_{lm} Z_m \right|_{\tilde{h}}^2} \iota^*_X( d \mu_{BZ}) \right)^{-1} \int_{ \prj^{N-1}} \frac{\tilde{h} (Z_i ,  {Z}_j) }{\sum_l \left| \sum_m B_{lm} Z_m \right|^2_{\tilde{h}}} \iota^*_X (d \mu_{BZ}).
\end{equation*}

\end{definition}

\begin{remark} \label{remhilbfspn}
Recalling (\ref{defoffseq}) in Definition \ref{defhilbfs}, we observe that the above can be written as $Q_X (B)_{ij} = Hilb(FS(H)) (Z_i , Z_j) / \tr(Hilb(FS(H)))$, with $H = \overline{(B^{-1})^t} B^{-1}$.
\end{remark}

As in Lemma \ref{lembdrpsizr}, we need to show that $Q_X$ extends continuously to the boundary, i.e.~$Q_X \in C^0 (\mathcal{J} , \mathcal{H})$. 

Recall the biholomorphic map $\xi_B: \prj^{N-1} \isom \prj^{N-1}$ induced from the linear action of $B \in \mathcal{J}^{\circ}$ on $\prj^{N-1}$, which moves $\iota (X) \subset \prj^{N-1}$ to $\xi_B \circ \iota (X)  \subset \prj^{N-1}$. Recall also that $\iota^*_X( d \mu_{BZ})$ is, as a measure on $X$, equal to $\iota^* \left( \omega^n_{ \widetilde{FS}(H)} \right)$. We then have
\begin{align*}
\int_{ \prj^{N-1}} \frac{\tilde{h} (Z_i ,  {Z}_j) }{\sum_l \left| \sum_m B_{lm} Z_m \right|^2_{\tilde{h}}} \iota^*_X (d \mu_{BZ}) 
&= \int_{\iota (X)  } \frac{\tilde{h} (Z_i ,  {Z}_j) }{\sum_l \left| \sum_m B_{lm} Z_m \right|^2_{\tilde{h}}} \omega^n_{\widetilde{FS}(H)} \\
&=    \sum_{r,s}  (B^*)^{-1}_{ri} B^{-1}_{js} \sum_{p,q} \int_{\iota (X) } \frac{\tilde{h} (B_{rp} Z_p ,  B_{sq} {Z}_q) }{\sum_l \left| \sum_m B_{lm} Z_m \right|^2_{\tilde{h}}} \omega^n_{\widetilde{FS}(H)} \\
&=    \sum_{r,s}  B^{-1}_{ri} B^{-1}_{js}  \int_{\xi_B \circ \iota (X) } \frac{\tilde{h} ( Z_r ,   {Z}_s) }{\sum_l \left| Z_l \right|^2_{\tilde{h}}} \omega^n_{\widetilde{FS}(H)} ,
\end{align*}
since $(\xi_B \circ \iota)^* (Z_i) = \sum_p B_{ip} \iota^* (Z_p)$. Writing $\Phi (B)$ for the matrix defined by 
\begin{equation*}
\Phi (B)_{rs} := \int_{\xi_B \circ \iota (X) } \frac{\tilde{h} ( Z_r ,   {Z}_s) }{\sum_l \left| Z_l \right|^2_{\tilde{h}}} \omega^n_{\widetilde{FS}(H)} ,
\end{equation*}
we have 
\begin{equation} \label{defofpsib}
	Q_X (B) = \frac{(B^t)^{-1} \Phi (B) (B^t)^{-1}}{\tr ((B^t)^{-1} \Phi (B) (B^t)^{-1})},
\end{equation}
analogously to (\ref{defofpsizrb}).

The following two lemmas are of crucial importance in the proof of the main surjectivity result (Proposition \ref{lemsurjhilb}).

\begin{lemma} \label{ctextqx}
	$Q_X$ continuously extends to the boundary of $\mathcal{J}$, and sends elements in $\partial \mathcal{J}$ to the ones in $\partial \mathcal{H}$, i.e.~$Q_X \in C^0 (\mathcal{J} , \mathcal{H} )$ with $Q_X (\partial \mathcal{J} ) \subset \partial \mathcal{H}$.
\end{lemma}

\begin{proof}
	
Suppose that $\{ B_{\nu} \}_{\nu}$ is any sequence in $\mathcal{J}^{\circ}$ converging to a point $B_{\infty}$ in $\partial \mathcal{J}$, in the Euclidean topology induced from $\mathcal{J}^{\circ} , \mathcal{H}^{\circ} \subset \mathbb{R}^{N^2 -1}$ (as in \S \ref{secprelimdef}). Each $B_{\nu}$ induces a holomorphic (in fact biholomorphic) map $\xi_{\nu} : \prj^{N-1} \to \prj^{N-1}$ given by the linear action of $B_{\nu}$ on $\prj^{N-1}$. In particular, for each $B_{\nu}$ we have an embedded variety $\xi_{\nu} \circ \iota (X) \subset \prj^{N-1}$.

We now consider the Hilbert scheme $\mathrm{Hilb}_{\prj^{N-1}} (n,P)$ of subvarieties in $\prj^{N-1}$ with dimension $n$ and the Hilbert polynomial $P(m) = \sum_{i=1}^n (-1)^i \dim H^i (X , \mathcal{L}^{\otimes m})$. Let $\{ p_{\nu} \}_{\nu}$ be the sequence of closed points in $\mathrm{Hilb}_{\prj^{N-1}} (n,P)$ defined by the embedded varieties $\{ \xi_{\nu} \circ \iota (X) \}_{\nu} \inj \prj^{N-1}$ (or equivalently, the homogeneous ideals that define these embedded varieties). Since $\mathrm{Hilb}_{\prj^{N-1}} (n,P)$ is a projective scheme \cite{groth}, it is proper and hence there exists a convergent subsequence $\{ p_{\nu (i)} \}_{i}$ in $\mathrm{Hilb}_{\prj^{N-1}} (n,P)$ with the limit $p_{\infty ,1}$, say (where the convergence is in terms of the analytic topology). Now choose another convergent subsequence $\{ p_{\nu (j)} \}_{j} \subset \mathrm{Hilb}_{\prj^{N-1}} (n,P)$ with the limit $p_{\infty ,2}$. Assume $p_{\infty ,1} \neq p_{\infty ,2}$. Then, since $\mathrm{Hilb}_{\prj^{N-1}} (n,P)$ is separated, by taking $\nu (i)$ and $\nu (j)$ to be sufficiently large, we can find open subsets $U_1$ and $U_2$ (in the analytic topology) in $\mathrm{Hilb}_{\prj^{N-1}} (n,P)$ such that $U_1 \cap U_2 = \emptyset$, $\{ p_{\nu (i)}  \}_{i} \subset U_1$, and $\{ p_{\nu (j)}  \}_{j} \subset U_2$. However, fixing a set of homogeneous polynomials that define $\iota (X) \inj \prj^{N-1}$ (i.e.~generators of the homogeneous ideal that define $\iota (X) \subset \prj^{N-1}$), $B_{\nu} \to B_{\infty}$ (as $\nu \to \infty$) implies that the homogeneous polynomials defining $p_{\nu} := \xi_{\nu} \circ \iota (X)$ must converge (in the sense that their coefficients converge). This contradicts $U_1 \cap U_2 = \emptyset$ with $\{ p_{\nu (i)}  \}_{i} \subset U_1$, $\{ p_{\nu (j)}  \}_{j} \subset U_2$, which was the definition of $U_1$ and $U_2$. Hence we get $p_{\infty ,1} = p_{\infty ,2}$, i.e.~the sequence $\{ p_{\nu} \}_{\nu}$ has a well-defined limit $p_{\infty} \in \mathrm{Hilb}_{\prj^{N-1}} (n,P)$.

Now we take the Chow scheme $\mathrm{Chow}_{\prj^{N-1}} (n,d)$ of algebraic cycles in $\prj^{N-1}$ with dimension $n$ and degree $d:=\deg \mathcal{L}$, and recall that there exists a universal morphism of schemes $f_{HC} : \mathrm{Hilb}_{\prj^{N-1}} (n,P) \to \mathrm{Chow}_{\prj^{N-1}} (n,d)$, called the Hilbert--Chow morphism, where $f_{HC} (p)$ is the algebraic cycle defined by $p \in \mathrm{Hilb}_{\prj^{N-1}} (n,P)$ (cf.~\cite[\S 5.4]{git}). We define the limit cycle $\displaystyle{ \lim_{\nu \to \infty} \xi_{\nu} \circ \iota (X)}$ to be $f_{HC} (p_{\infty})$. Since $f_{HC}$ is a morphism of schemes, we have $\displaystyle{f_{HC} (p_{\infty}) = \lim_{\nu \to \infty} f_{HC} (p_{\nu})}$ in $\mathrm{Chow}_{\prj^{N-1}} (n,d)$, by continuity of $f_{HC}$ in the analytic topology.

By the definition of the Hilbert scheme limit $p_{\infty}$, we have
\begin{equation*}
	\Phi \left( \lim_{\nu \to \infty} B_{\nu} \right)_{rs} = \int_{f_{HC} (p_{\infty})} \frac{\tilde{h} ( Z_r ,   {Z}_s) }{\sum_l \left| Z_l \right|^2_{\tilde{h}}} \omega^n_{\widetilde{FS}(H)}  = \int_{\displaystyle{ \lim_{\nu \to \infty} \xi_{\nu} \circ \iota (X) }} \frac{\tilde{h} ( Z_r ,   {Z}_s) }{\sum_l \left| Z_l \right|^2_{\tilde{h}}} \omega^n_{\widetilde{FS}(H)} .
\end{equation*}
On the other hand, we also have
\begin{align*}
	\int_{ f_{HC} (p_{\infty})} \frac{\tilde{h} ( Z_r ,   {Z}_s) }{\sum_l \left| Z_l \right|^2_{\tilde{h}}} \omega^n_{\widetilde{FS}(H)} 
	&= \lim_{\nu \to \infty} \int_{f_{HC} (p_{\nu})} \frac{\tilde{h} ( Z_r ,   {Z}_s) }{\sum_l \left| Z_l \right|^2_{\tilde{h}}} \omega^n_{\widetilde{FS}(H)} \\
	&= \lim_{\nu \to \infty} \int_{\xi_{\nu} \circ \iota (X)} \frac{\tilde{h} ( Z_r ,   {Z}_s) }{\sum_l \left| Z_l \right|^2_{\tilde{h}}} \omega^n_{\widetilde{FS}(H)} = \lim_{\nu \to \infty} \Phi (B_{\nu}),
\end{align*}
by $\displaystyle{f_{HC} (p_{\infty}) = \lim_{\nu \to \infty} f_{HC} (p_{\nu})}$, which follows from the continuity of $f_{HC}$. We thus get
\begin{equation} \label{cttlph}
\Phi \left( \displaystyle{\lim_{\nu \to \infty}} B_{\nu} \right) = \displaystyle{\lim_{\nu \to \infty}} \Phi (B_{\nu}) .	
\end{equation}
Observe that this is positive semidefinite hermitian, as it is a limit of positive definite hermitian matrices.

We now prove that the limit $Q_X \left( \displaystyle{\lim_{\nu \to \infty}} B_{\nu} \right)$ is a well-defined element in $\partial \mathcal{H}$. As we did in the proof of Lemma \ref{lembdrpsizr}, we write $\lambda_{\nu}$ for the smallest eigenvalue of $B_{\nu}$ and observe, by homogeneity, that
\begin{align*}
	 Q_X \left( B_{\nu} \right) &=  \frac{(B^t_{\nu})^{-1} \Phi (B_{\nu}) (B^t_{\nu})^{-1}}{ \tr ((B^t_{\nu})^{-1} \Phi (B_{\nu}) (B^t_{\nu})^{-1})}, \\
	 &=  \frac{\lambda_{\nu} (B^t_{\nu})^{-1} \Phi (B_{\nu})  \lambda_{\nu} (B^t_{\nu})^{-1}}{ \tr ( \lambda_{\nu} (B^t_{\nu})^{-1} \Phi (B_{\nu}) \lambda_{\nu} (B^t_{\nu})^{-1})}.
\end{align*}
Thus, writing $\beta_{\infty} = \displaystyle{\lim_{\nu \to \infty}} \lambda_{\nu} ( B^t_{\nu} )^{-1}$, which is strictly positive semidefinite (as we saw in Lemma \ref{lemconvbeta}), we get
\begin{equation*}
	 Q_X \left( \lim_{\nu \to \infty} B_{\nu} \right) = \frac{\displaystyle{ \beta_{\infty} \Phi \left(  \lim_{\nu \to \infty}  B_{\nu} \right) \beta_{\infty}}}{ \displaystyle{ \tr \left( \beta_{\infty} \Phi \left( \lim_{\nu \to \infty}  B_{\nu} \right) \beta_{\infty} \right)}}.
\end{equation*}
Combining this equality with (\ref{cttlph}), we get
\begin{equation*}
	 Q_X \left( \lim_{\nu \to \infty} B_{\nu} \right) = \frac{ \displaystyle{ \beta_{\infty} \lim_{\nu\to \infty} \Phi (  B_{\nu} ) \beta_{\infty}}}{ \displaystyle{\tr \left( \beta_{\infty}  \lim_{\nu\to \infty} \Phi ( B_{\nu}) \beta_{\infty} \right)}} = \lim_{\nu \to \infty} Q_X(B_{\nu}).
\end{equation*}
Moreover, since $\beta_{\infty}$ is strictly positive semidefinite (i.e.~positive semidefinite with rank $1 \le \mathrm{rk} (\beta_{\infty}) \le N-1$) and $ \displaystyle{ \lim_{\nu \to \infty} \Phi (  B_{\nu} )}$ is positive semidefinite, we see that $Q_X \left( \displaystyle{\lim_{\nu \to \infty}} B_{\nu} \right)$ is a strictly positive semidefinite hermitian matrix, and hence is an element in $\partial \mathcal{H}$.

Since any element in $\partial \mathcal{J}$ can be realised as a limit of a sequence $\{ B_{\nu} \}_{\nu} \subset \mathcal{J}^{\circ}$, $Q_X$ continuously extends to the boundary of $\mathcal{J}$, and sends elements in $\partial \mathcal{J}$ to the ones in $\partial \mathcal{H}$.
\end{proof}

We now define a 1-parameter family of continuous maps $Q_t : \mathcal{J}^{\circ} \to \mathcal{H}^{\circ}$, $t \in [0,1]$, by
\begin{equation*}
Q_t (B) : = t Q_X (B) + (1-t) Q_{\prj} (B).
\end{equation*}

\begin{lemma} \label{ctextqt}
	$Q_t$ continuously extends to the boundary of $\mathcal{J}$, and sends elements in $\partial \mathcal{J}$ to the ones in $\partial \mathcal{H}$ for all $t \in [0,1]$, i.e.~$Q_t \in C^0 (\mathcal{J} , \mathcal{H} )$ with $Q_t (\partial \mathcal{J} ) \subset \partial \mathcal{H}$ for all $t \in [0,1]$.
\end{lemma}

\begin{proof}
	We argue similarly to the proof of Lemma \ref{ctextqx}. Observe first that we can write
\begin{equation*}
	Q_t (B) =(B^t)^{-1} \left(  \frac{t \Phi (B)}{\tr ((B^t)^{-1} \Phi (B) (B^t)^{-1} )} +  \frac{(1-t) Q_{\prj} (\mathrm{Id})}{\tr( (B^t)^{-1} Q_{\prj} (B^t)^{-1})}  \right) (B^t)^{-1} ,
\end{equation*}
by recalling (\ref{defofpsizrb}) and (\ref{defofpsib}). Given a sequence $\{ B_{\nu} \}_{\nu} \subset \mathcal{J}^{\circ}$ converging to $B_{\infty} \in \partial \mathcal{J}$ as before, we get (again by homogeneity)
\begin{equation*}
	Q_t (B_{\nu}) = \lambda_{\nu} (B^t)^{-1} \left(  \frac{t \Phi (B_{\nu})}{\tr (\lambda_{\nu} (B_{\nu}^t)^{-1} \Phi (B_{\nu}) \lambda_{\nu} (B_{\nu}^t)^{-1} )} +  \frac{(1-t) Q_{\prj} (\mathrm{Id})}{\tr( \lambda_{\nu} (B_{\nu}^t)^{-1} Q_{\prj} (\mathrm{Id}) \lambda_{\nu} (B_{\nu}^t)^{-1})}  \right) \lambda_{\nu} (B_{\nu}^t)^{-1} ,
\end{equation*}
where $\lambda_{\nu}$ is the smallest eigenvalue of $B_{\nu}$, as in the proof of Lemma \ref{lembdrpsizr} or \ref{ctextqx}. Thus, we get
\begin{align}
	Q_t \left( \lim_{\nu \to \infty} B_{\nu} \right) &= \beta_{\infty} \left(  \frac{ \displaystyle{ t \lim_{\nu\to \infty} \Phi ( B_{\nu})}}{\displaystyle{ \tr \left( \beta_{\infty}  \lim_{\nu\to \infty} \Phi ( B_{\nu}) \beta_{\infty} \right)}} +  \frac{(1-t) Q_{\prj} (\mathrm{Id})}{\tr( \beta_{\infty} Q_{\prj} (\mathrm{Id}) \beta_{\infty})}  \right) \beta_{\infty}, \label{btinfctph} \\
	&=\lim_{\nu \to \infty} Q_t(B_{\nu}), \notag
\end{align}
where $\beta_{\infty} =  \displaystyle{\lim_{\nu \to \infty}} \lambda_{\nu} ( B^t_{\nu} )^{-1}$ as in Lemma \ref{lemconvbeta}, and we used (\ref{cttlph}).

Since the terms in the bracket of (\ref{btinfctph}) is positive semidefinite and $\beta_{\infty}$ is strictly positive semidefinite (i.e.~positive semidefinite with rank $1 \le \mathrm{rk} (\beta_{\infty}) \le N-1$), we conclude that $Q_t \left( \displaystyle{\lim_{\nu \to \infty}} B_{\nu} \right)$ is a strictly positive semidefinite hermitian matrix for all $t \in [0,1]$, which completes the proof.
\end{proof}

\begin{proposition} \label{lemsurjhilb}
Suppose that $\mathcal{L}$ is very ample. Then $Hilb : \mathcal{H} (X , \mathcal{L}) \to \mathcal{B}$ is surjective.	
\end{proposition}

\begin{proof}
We now recall $Q_{\prj}$ as defined in (\ref{defofpsizrb}). By Lemmas \ref{lempsizediffoss} and \ref{lembdrpsizr}, $Q_{\prj}$ defines a diffeomorphism between $\mathcal{J}^{\circ}$ and $\mathcal{H}^{\circ}$ which continuously extends to a map $\mathcal{J} \to \mathcal{H}$ with $Q_{\prj} (\partial \mathcal{J}) \subset \partial \mathcal{H}$. Since $Q_{\prj} : \mathcal{J}^{\circ} \to \mathcal{H}^{\circ}$ is a diffeomorphism, we have
\begin{equation} \label{degnzpszr}
	\mathrm{deg} (Q_{\prj} , \mathcal{J}^{\circ} , p) \neq 0
\end{equation}
for all $p \in \mathcal{H}^{\circ}$ by (\ref{defdegcone}); note $\mathcal{H}^{\circ} \cap Q_{\prj} (\partial \mathcal{J}) = \emptyset$.

We now pick an arbitrary $p \in \mathcal{H}^{\circ}$ and compute the degree $\mathrm{deg}(Q_X , \mathcal{J}^{\circ} , p)$. Since Lemma \ref{ctextqt} shows $Q_t (\partial \mathcal{J}) \subset \partial \mathcal{H}$ for all $t \in [0,1]$, we can apply Theorem \ref{htpyinvdeg} to prove
\begin{equation*}
	\mathrm{deg}(Q_X , \mathcal{J}^{\circ} , p) = \mathrm{deg}(Q_{\prj} , \mathcal{J}^{\circ} , p) \neq 0 ,
\end{equation*}
by also recalling (\ref{degnzpszr}). Since $p \in \mathcal{H}^{\circ}$ and $\mathcal{H}^{\circ} \cap Q_X(\partial \mathcal{J} ) = \emptyset$ by Lemma \ref{ctextqx}, we can apply Theorem \ref{nontrvdegsurj} to prove that there exists $x \in \mathcal{J}^{\circ}$ such that $Q_X(x) = p$. We thus conclude that $Q_X$ is surjective.

Finally, we recall that $\iota^*_X( d \mu_{BZ}) = \iota^* \left( \omega^n_{ \widetilde{FS}(H)} \right)$ is equal to $ \omega^n_{FS(H)} $ as a volume form on $X$. Note also that, writing $h$ for $\iota^* \tilde{h} = \iota^* \tilde{h}_{FS( \mathrm{Id} )}$, we can re-write the definition of $Q_X (B)$ (cf.~Definition \ref{defqx}) as
\begin{equation} \label{psibitofsb}
Q_X (B)_{ij} =  \left( \int_X \frac{\sum_l |s_l|_{h}^2}{\sum_l \left| \sum_m B_{lm} s_m \right|_{h}^2} \omega^n_{FS(H)} \right)^{-1} \int_{X} \frac{h (s_i ,  s_j) }{\sum_l \left| \sum_m B_{lm} s_m \right|^2_{h}} \omega^n_{FS(H)} .
\end{equation}
where we wrote $s_i := \iota^*Z_i$. Recalling Definition \ref{defhilbfs} (and also Remark \ref{remhilbfspn}), we observe that this can be re-written as
\begin{equation*}
Q_X (B)_{ij} =  \left( \int_X \sum_l |s_l|_{FS(H)}^2 \omega^n_{FS(H)} \right)^{-1} \int_{X} FS(H) (s_i ,  s_j)  \omega^n_{FS(H)} ,
\end{equation*}
which is a (positive) constant multiple of $Hilb(FS(H)) (s_i , s_j)$. Thus, the surjectivity of $Q_X$ that we proved implies the following consequence: fixing a basis $\{ s_i \}_i$ for $H^0 (X,\mathcal{L})$, for any positive definite hermitian matrix $G$ there exists $H \in \mathcal{B}$ and a constant $\alpha$ such that $Hilb(FS(H)) (s_i , s_j) = e^{\alpha} G_{ij}$, or equivalently
\begin{equation*}
	Hilb(e^{ - \alpha} FS(H)) (s_i , s_j) =  G_{ij} .
\end{equation*}
Since $\alpha$ is a constant, $e^{ - \alpha} FS(H)$ is positively curved (with the curvature $\omega_{FS(H)} >0$) and hence defines an element of $\mathcal{H} (X,\mathcal{L})$. Thus, the above establishes the required statement that $Hilb$ is surjective.
\end{proof}

\begin{remark}
We provide another point of view regarding the equation (\ref{psibitofsb}). Fixing a hermitian metric $h \in \mathcal{H} (X,\mathcal{L})$, we observe that there exists $\beta \in C^{\infty} (X , \rl)$ such that $\omega^n_{FS(H)} = e^{\beta} \omega^n_h$. Writing (\ref{psibitofsb}) in terms of $h$, the surjectivity of $Q_X$ can be regarded as ensuring the existence of $\phi \in C^{\infty} (X, \rl)$ such that
\begin{equation} \label{surjfvbtphg}
\int_X e^{\beta + \phi} h (s_i , s_j) \omega^n_h = G_{ij},
\end{equation}
for any given positive definite hermitian matrix $G_{ij}$.

The left hand side of the above equation is a priori not a $Hilb$ of a hermitian metric, because $e^{\beta + \phi} h$ is used for the metric on $\mathcal{L}$ whereas $\omega_h$ is used for the volume form. However, we can always find a function $f \in C^{\infty} (X, \rl)$, such that $e^{-f} h$ is positively curved and $Hilb (e^{-f} h) (s_i , s_j) = \int_{X} e^{\beta + \phi} h (s_i , s_j) \omega^n_h$, for \textit{any} $\beta$ and $\phi$; indeed, for this purpose, it is sufficient to solve for $f$ the following nonlinear PDE:
\begin{equation*}
 \left( \omega_h + \frac{\ai}{2 \pi } \ddbar f \right)^n = e^{f + \beta + \phi} \omega_h^n ,
\end{equation*}
which is solvable by the Aubin--Yau theorem (cf.~\cite{aubin} and \cite[Theorem 4, p.383]{yaucy}).

Thus, to prove the surjectivity of $Hilb$, it suffices to find \textit{some} $\beta , \phi \in C^{\infty} (X, \rl)$ that satisfy (\ref{surjfvbtphg}) with respect to the fixed $h \in \mathcal{H} (X,\mathcal{L})$, which is weaker than the surjectivity of $Hilb$ itself.
\end{remark}

\subsection{Variants of the Hilbert map} \label{secvarhilb}

We now recall that there are several variants of the Hilbert map that also appear in the literature \cite{bbgz,bw,donnum,kms,saitak}. We define the $Hilb_{ \nu}$ map
\begin{equation*}
	Hilb_{ \nu} (h) := \int_X h (,) d\nu
\end{equation*}
where the volume form $d \nu$ is one of the following.

\begin{enumerate}
	\item $d \nu$ is a \textit{fixed} volume form on $X$; an example of this is when $X$ is Calabi--Yau, in which case we can use the holomorphic volume form $\Omega \in H^0(X,K_X)$ to define $d \nu := \Omega \wedge \bar{\Omega}$,
	\item $d \nu$ is \textit{anticanonical}; a hermitian metric $h$ on $-K_X$ defines a volume form $d \nu^{ac} (h)$, where we note $d \nu^{ac} (e^{-\varphi} h) = e^{-\varphi} d \nu^{ac} (h)$,
	\item $d \nu$ is \textit{canonical}; a hermitian metric $h$ on $K_X$ defines a dual metric on $-K_X$, which defines a volume form $d \nu^{c} (h)$ with $d \nu^{c} (e^{-\varphi} h) = e^{\varphi} d \nu^{c} (h)$.
\end{enumerate}

We prove the following analogue of Proposition \ref{lemsurjhilb}. As we can see from the statement below, we need to consider the higher tensor power $\mathcal{L}^{\otimes k}$ and the vector space $H^0 (X,\mathcal{L}^{\otimes k})$ of dimension $N_k$.
\begin{proposition} \label{lemsurjhilbnu}
	Let $Hilb_{\nu}$ be defined by one of the three volume forms as defined above, and let $\mathcal{L}^{\otimes k}$ be a very ample line bundle. Then any positive definite hermitian matrix on $H^0 (X,\mathcal{L}^{\otimes k})$ can be realised as
\begin{equation*}
	Hilb_{ \nu} (\tilde{h}^{\otimes k}) :=  \int_X \tilde{h}^{\otimes k} (,) d\nu
\end{equation*}
	for some hermitian metric $\tilde{h}$ on $\mathcal{L}$ if
	\begin{enumerate}
		\item $\mathcal{L}$ is any ample line bundle and $d \nu$ is any fixed volume form,
		\item $\mathcal{L} = -K_X$ is ample with the anticanonical volume form $d \nu^{ac} (\tilde{h})$,
		\item $\mathcal{L} = K_X$ is ample with the canonical volume form $d \nu^{c} (\tilde{h})$.
	\end{enumerate}
\end{proposition}

\begin{proof}
Fixing a basis $\left\{ s_i \right\}_i$ and a hermitian metric $h^{\otimes k}$ on $\mathcal{L}^{\otimes k}$, (\ref{surjfvbtphg}) implies that for any positive definite hermitian matrix $G_{ij}$ there exists $\varphi \in C^{\infty} (X, \rl)$ such that
\begin{equation*}
	 \int_X e^{\varphi} h^{\otimes k} (s_i , s_j) \omega^n_{h} = G_{ij} ,
\end{equation*}
(we used $\omega^n_{h}$ for the volume form instead of $\omega^n_{h^{\otimes k}}$, but the difference is just a constant multiple which we can absorb in $\varphi$).

Observe that for each of the three choices $d \nu$, $d \nu^{ac}(h)$, $d \nu^{c} (h)$ of the volume form $d \nu$, there exists a function $\phi \in C^{\infty} (X, \rl)$ such that
\begin{equation*}
	\omega^n_h  = e^{\phi} d \nu,
\end{equation*}
and hence the claim follows from the following;
\begin{enumerate}
	\item $d \nu$ fixed: take $\tilde{h} := \exp \left( \frac{1}{k} (\varphi + \phi  ) \right) h$ so that $\tilde{h}$ satisfies $Hilb_{ \nu} (\tilde{h}) (s_i , s_j) = G_{ij}$;
	\item $d \nu$ anticanonical: take $\tilde{h} := \exp \left( \frac{1}{k+1} (\varphi + \phi  ) \right) h$ so that $Hilb_{ \nu} (\tilde{h}) (s_i , s_j) = G_{ij}$ with $d \nu = d \nu^{ac} (\tilde{h})$;
	\item $d \nu$ canonical: take $\tilde{h} := \exp \left( \frac{1}{k-1} (\varphi + \phi  ) \right) h$ so that $Hilb_{ \nu} (\tilde{h}) (s_i , s_j) = G_{ij}$ with $d \nu = d \nu^{c} (\tilde{h})$.
\end{enumerate}
\end{proof}

\begin{remark}
Unlike the case of the usual $Hilb$ as treated in Proposition \ref{lemsurjhilb}, the above proof does not show that $\tilde{h}$ has positive curvature; the associated curvature form $\omega_{\tilde{h}}$ may not be a \kah metric.
\end{remark}

\section{Injectivity of $FS$}

We establish the following ``quantitative injectivity'' to prove the second part of Theorem \ref{hilbsjfsinj}.
\begin{lemma} \label{lemfsinj}
Suppose that $\mathcal{L}$ is very ample, and that $H , H' \in \mathcal{B}$ satisfy 
\begin{equation*}
FS(H) = (1 +f) FS(H')
\end{equation*}
with $\sup_X |f| \le \epsilon$ for $\epsilon \ge 0$ satisfying $N^{\frac{3}{2}} \epsilon \le 1/4$.

Then we have $|| H - H' ||_{op} \le 4 N^{\frac{3}{2}} \epsilon$, where $|| \cdot ||_{op}$ is the operator norm, i.e.~the maximum of the moduli of the eigenvalues. In particular, considering the case $\epsilon=0$, we see that $FS$ is injective.
\end{lemma}

\begin{proof}
We now pick an $H$-orthonormal basis $\{ s_i \}_i$ and represent $H$ (resp.~$H'$) as a matrix $H_{ij}$ (resp.~$H'_{ij}$) with respect to the basis $\{ s_i \}_i$. $H_{ij}$ is the identity matrix, and replacing $\{ s_i \}_i$ by an $H$-unitarily equivalent basis if necessary, we may further assume $H'_{ij} = \textup{diag} (d^2_1 , \dots , d^2_N)$ for some $d_i >0$. Recall that the equation (\ref{defoffseq}) implies that we can write $FS(H') = e^{- \varphi} FS(H)$ with $\varphi = \log \left( \sum_{i=1}^{N} d^{-2}_i |s_i|^2_{FS(H)} \right) $. Thus the equation $FS(H) = (1 +f) FS(H')$ implies $1+f = \sum_i d_i^{-2} |s_i|^2_{FS(H)}$, and hence, by recalling (\ref{defoffseq}),
\begin{equation} \label{injfslineq}
(1 + f) \sum_i |s_i|_{h}^2 = \sum_i d_i^{-2} |s_i|_{h}^2 , 
\end{equation}
with respect to \textit{any} hermitian metric $h$ on $\mathcal{L}$, by noting that we may multiply both sides of (\ref{injfslineq}) by any strictly positive function $e^{\phi}$. We now fix this basis $\{ s_i \}_i$, and the operator norm $|| \cdot ||_{op}$ or the Hilbert--Schmidt norm $|| \cdot ||_{HS}$ used in this proof will all be computed with respect to this basis.

We now choose $N$ hermitian metrics $h_1 , \dots , h_N$ on $\mathcal{L}$ as follows. Recall now that, by Proposition \ref{lemsurjhilb}, for any $N$-tuple of strictly positive numbers $\vec{\lambda} = (\lambda_1 , \dots , \lambda_N)$ there exists $\phi_{\vec{\lambda}} \in C^{\infty} (X , \rl)$ such that the hermitian metric $h' := \exp (\phi_{\vec{\lambda}})h$ satisfies $\int_X |s_i|^2_{h'} \omega^n_{h'} = \lambda_i$. We thus take 
\begin{equation*}
	\vec{\lambda}_i = (\delta, \dots ,  \delta, 1 , \delta, \dots , \delta)
\end{equation*}
with $1$ in the $i$-th place, where $0 < \delta \ll 1$ is chosen to be small enough so that the matrix defined as
\begin{equation*}
\Lambda := 
\begin{pmatrix}
\vec{\lambda}_1 \\
 \vdots \\
  \vec{\lambda}_N
\end{pmatrix}
\end{equation*}
satisfies $||\Lambda||_{op} \le 2$ and $|| \Lambda^{-1} ||_{op} \le 2$. (How small $\delta$ must be depends on $N$, but this will not concern us.)

We now choose $\phi_i \in C^{\infty} (X , \rl)$ appropriately (cf.~Proposition \ref{lemsurjhilb}) so that $h_i := \exp (\phi_i) h$ satisfies
\begin{equation*}
\vec{\lambda}_i =
\left( \int_X |s_1|^2_{h_i} \omega^n_{h_i} , \int_X |s_2|^2_{h_i} \omega^n_{h_i} , \dots , \int_X |s_N|^2_{h_i} \omega^n_{h_i} \right).
\end{equation*}
Then, multiplying both sides of (\ref{injfslineq}) by $\exp (\phi_i)$ and integrating over $X$ with respect to the measure $\omega_{h_i}^n / n!$, we get the following system of linear equations
\begin{equation*}
(\Lambda +F) 
\begin{pmatrix}
1 \\
  \vdots \\
   1
\end{pmatrix}
=
\Lambda
\begin{pmatrix}
d^{-2}_1 \\
  \vdots \\
   d^{-2}_N
\end{pmatrix} ,
\end{equation*}
where $F$ is a matrix defined by
\begin{equation*}
F_{ij} := \int_X f |s_j|^2_{h_i} \omega^n_{h_i}
\end{equation*}
whose max norm (i.e.~the maximum of the moduli of its entries) satisfies $||F||_{max} \le \sup_X |f| \le \epsilon$ since the modulus of each entry of $\Lambda$ is at most $1$. We thus get
\begin{equation*}
\begin{pmatrix}
d^{-2}_1 -1\\
  \vdots \\
   d^{-2}_N -1
\end{pmatrix}
=
  \Lambda^{-1} F 
\begin{pmatrix}
1 \\
  \vdots \\
   1
\end{pmatrix}.
\end{equation*}
Thus, noting $||  \Lambda^{-1} F ||_{op} \le ||\Lambda^{-1}||_{op} ||F||_{op} \le 2 ||F||_{HS} \le 2 N ||F||_{max} \le 2N \epsilon$, we get
\begin{equation*}
| d_i^{-2} -1 |\le \sqrt{\sum_i |d^{-2}_i -1|^2} \le 2 N^{1 + \frac{1}{2}} \epsilon .
\end{equation*}
Thus we get $1 - 2N^{\frac{3}{2}} \epsilon \le d_i^{-2} \le 1 + 2N^{\frac{3}{2}} \epsilon$, and by the assumption $N^{\frac{3}{2}} \epsilon \le 1/4$ we have
\begin{equation*}
	1 - 4 N^{\frac{3}{2}} \epsilon < 1 - \frac{2N^{\frac{3}{2}} \epsilon}{1+ 2N^{\frac{3}{2}}\epsilon} \le d_i^{2} \le 1 + \frac{2N^{\frac{3}{2}} \epsilon}{1- 2N^{\frac{3}{2}}\epsilon}  < 1 + 4 N^{\frac{3}{2}} \epsilon
\end{equation*}
as required.
\end{proof}

\begin{remark} \label{remscalfsquan}
	We recast Lemma \ref{lemfsinj} in the case we use the usual scaling convention for the Fubini--Study map (cf.~Remark \ref{remscalhbfs}). Suppose that $\mathcal{L}^{\otimes k}$ is very ample and write $\mathcal{B}_k$ for the space of hermitian forms on the vector space $H^0 (X , \mathcal{L}^{\otimes k})$ of dimension $N_k$. The statement for this convention is as follows: if $H , H' \in \mathcal{B}_k$ satisfy $FS(H)^{\otimes k} = (1 +f) FS(H')^{\otimes k}$ with $\sup_X |f| \le \epsilon$ for $0< N_k^{\frac{3}{2}} \epsilon \le 1/4$, then we have $|| H - H' ||_{op} \le 4 N_k^{\frac{3}{2}} \epsilon$.
\end{remark}


\begin{thebibliography}{1}

\bibitem{Finski22}
Siarhei Finski.
\newblock {O}n the metric structure of section ring.
\newblock {\em arXiv preprint arXiv:2209.03853}, 2022.

\bibitem{yhqifs}
Yoshinori Hashimoto.
\newblock {Q}uantitative injectivity of the {F}ubini--{S}tudy map.
\newblock {\em arXiv preprint arXiv:2502.08038}, 2025.

\bibitem{Lem21}
L{\'a}szl{\'o} Lempert.
\newblock {O}n the {B}ergman kernels of holomorphic vector bundles.
\newblock {\em arXiv preprint arXiv:2109.08593}, 2021.

\bibitem{Sun22}
Jingzhou Sun.
\newblock {O}n {T}he {I}mage {O}f {T}he {H}ilbert {M}ap.
\newblock {\em arXiv preprint arXiv:2208.13407}, 2022.

\end{thebibliography}

\begin{thebibliography}{10}

\bibitem{amd}
Ravi~P. Agarwal, Maria Meehan, and Donal O'Regan, \emph{Fixed point theory and
  applications}, Cambridge Tracts in Mathematics, vol. 141, Cambridge
  University Press, Cambridge, 2001. \MR{1825411 (2002c:47122)}

\bibitem{aubin}
Thierry Aubin, \emph{{\'E}quations du type {M}onge-{A}mp\`ere sur les
  vari\'et\'es k\"ahl\'eriennes compactes}, Bull. Sci. Math. (2) \textbf{102}
  (1978), no.~1, 63--95. \MR{494932}

\bibitem{bbgz}
Robert~J. Berman, S\'ebastien Boucksom, Vincent Guedj, and Ahmed Zeriahi,
  \emph{A variational approach to complex {M}onge-{A}mp\`ere equations}, Publ.
  Math. Inst. Hautes \'Etudes Sci. \textbf{117} (2013), 179--245. \MR{3090260}

\bibitem{bw}
Robert~J Berman and David~Witt Nystrom, \emph{Complex optimal transport and the
  pluripotential theory of {K}\"{a}hler-{R}icci solitons}, arXiv preprint
  arXiv:1401.8264 (2014).

\bibitem{bly}
Jean-Pierre Bourguignon, Peter Li, and Shing-Tung Yau, \emph{Upper bound for
  the first eigenvalue of algebraic submanifolds}, Comment. Math. Helv.
  \textbf{69} (1994), no.~2, 199--207. \MR{1282367 (95j:58168)}

\bibitem{donproj2}
Simon~K. Donaldson, \emph{Scalar curvature and projective embeddings. {II}}, Q.
  J. Math. \textbf{56} (2005), no.~3, 345--356. \MR{2161248 (2006f:32033)}

\bibitem{donnum}
\bysame, \emph{Some numerical results in complex differential geometry}, Pure
  Appl. Math. Q. \textbf{5} (2009), no.~2, Special Issue: In honor of Friedrich
  Hirzebruch. Part 1, 571--618. \MR{2508897 (2010d:32020)}

\bibitem{groth}
Alexander Grothendieck, \emph{Techniques de construction et th\'eor\`emes
  d'existence en g\'eom\'etrie alg\'ebrique. {IV}. {L}es sch\'emas de
  {H}ilbert}, S\'eminaire {B}ourbaki, {V}ol.\ 6, Soc. Math. France, Paris,
  1995, pp.~Exp.\ No.\ 221, 249--276. \MR{1611822}

\bibitem{yhpreprint}
Yoshinori Hashimoto, \emph{Quantisation of extremal {K}\"{a}hler metrics},
  arXiv preprint arXiv:1508.02643 (2015).

\bibitem{kms}
Julien Keller, Julien Meyer, and Reza Seyyedali, \emph{Quantization of the
  {L}aplacian operator on vector bundles, {I}}, Math. Ann. \textbf{366} (2016),
  no.~3-4, 865--907. \MR{3563226}

\bibitem{git}
David Mumford, John Fogarty, and Frances Kirwan, \emph{Geometric invariant
  theory}, third ed., Ergebnisse der Mathematik und ihrer Grenzgebiete (2)
  [Results in Mathematics and Related Areas (2)], vol.~34, Springer-Verlag,
  Berlin, 1994. \MR{1304906}

\bibitem{saitak}
Shunsuke Saito and Ryosuke Takahashi, \emph{Stability of anti-canonically
  balanced metrics}, arXiv preprint arXiv:1607.05534v2 (2017).

\bibitem{yaucy}
Shing~Tung Yau, \emph{On the {R}icci curvature of a compact {K}\"ahler manifold
  and the complex {M}onge-{A}mp\`ere equation. {I}}, Comm. Pure Appl. Math.
  \textbf{31} (1978), no.~3, 339--411. \MR{480350 (81d:53045)}

\end{thebibliography}

\providecommand{\bysame}{\leavevmode\hbox to3em{\hrulefill}\thinspace}
\providecommand{\MR}{\relax\ifhmode\unskip\space\fi MR }
\providecommand{\MRhref}[2]{%
  \href{http://www.ams.org/mathscinet-getitem?mr=#1}{#2}
}
\providecommand{\href}[2]{#2}

\begin{flushleft}
Dipartimento di Matematica e Informatica ``U.~Dini'', \\
Universit\`a degli Studi di Firenze, \\
Viale Morgagni 67/A 50134 Firenze, Italy. \\

Email: \verb|yoshinori.hashimoto@unifi.it|
\end{flushleft}

\end{document}